\documentclass[11pt]{amsart}
\usepackage[utf8]{inputenc}
\usepackage[margin=1in]{geometry}
\usepackage{amssymb, amsmath, amsthm, multirow}
\usepackage{tikz-cd}
\usetikzlibrary{matrix,arrows,decorations.pathmorphing}
\usepackage[pagebackref]{hyperref}
\hypersetup{breaklinks, colorlinks, linkcolor=magenta}

\newtheorem{thm}{Theorem}[section]
\newtheorem{lem}[thm]{Lemma}
\newtheorem{cor}[thm]{Corollary}
\newtheorem{prop}[thm]{Proposition}
\newtheorem{ques}{Question}
\newtheorem{theorem}{Theorem}

\theoremstyle{definition}
\newtheorem{defn}[thm]{Definition}

\theoremstyle{remark}
\newtheorem{rem}[thm]{Remark}
\newtheorem{ex}[thm]{Example}

\def\P{\mathbb{P}}
\def\Z{\mathbb{Z}}
\def\G{\mathbb{G}}
\def\C{\mathcal{C}}

\def\O{\mathcal{O}}
\def\<{\langle}
\def\>{\rangle}

\DeclareMathOperator{\GL}{GL}
\DeclareMathOperator{\SL}{SL}
\DeclareMathOperator{\Hom}{Hom}
\DeclareMathOperator{\Aut}{Aut}
\DeclareMathOperator{\Gal}{Gal}
\DeclareMathOperator{\op}{op}

\DeclareMathOperator{\End}{End}
\DeclareMathOperator{\Spec}{Spec}
\DeclareMathOperator{\M}{M}

\DeclareMathOperator{\Pic}{Pic}
\DeclareMathOperator{\CDiv}{CDiv}
\DeclareMathOperator{\Br}{Br}
\DeclareMathOperator{\R}{R}

\DeclareMathOperator{\Bl}{Bl}
\DeclareMathOperator{\Res}{Res}
\DeclareMathOperator{\Ext}{Ext}
\DeclareMathOperator{\perf}{perf}

\title{Toric surfaces over an arbitrary field}
\author{Fei Xie}
\address{Fakult\"{a}t f\"{u}r Mathematik, Universit\"{a}t Bielefeld, Bielefeld, 33615, Germany}
\email{fxie@math.uni-bielefeld.de}
\thanks{The author was supported in part by the NSF research grant DMS-\#1160206.}

\begin{document}
\maketitle
\begin{abstract}
We study toric varieties over an arbitrary field with an emphasis on toric surfaces in the Merkurjev-Panin motivic category of ``K-motives". We explore the decomposition of certain toric varieties as K-motives into products of central simple algebras, the geometric and topological information encoded in these central simple algebras, and the relationship between the decomposition of the K-motives and the semiorthogonal decomposition of the derived categories. We obtain the information mentioned above for toric surfaces by explicitly classifying all minimal smooth projective toric surfaces using toric geometry.
\end{abstract}
\section{Introduction}
Throughout, we fix an arbitrary base field $k$. Let $X$ be a scheme over $k$ and let $K/k$ be a field extension. We say a scheme $Y$ over $k$ is a $K/k$-\textit{form} of $X$ if the schemes $X_K:= X\otimes_k K$ and $Y_K$ are isomorphic as schemes over $K$ \cite[Chapter III $\S 1$]{galoiscoh}. Let $k^s$ be the separable closure of $k$. A $k^s/k$-form is simply called a \textit{form} or \textit{twisted form}. The scheme $X_{k^s}$ has a natural $\Gamma=\Gal(k^s/k)$-action.

We will focus on the study of toric varieties over $k$. Let $X$ be a normal geometrically irreducible variety over $k$ and let $T$ be an algebraic torus acting on $X$ over $k$. The variety $X$ is a toric $T$-variety if there is an open orbit $U$ such that $U$ is a principal homogeneous space or torsor over $T$. A toric $T$-variety is called split if the torus $T$ is split. The case of split toric varieties have been extensively studied, for example in \cite{dantoric}\cite{ftoric}\cite{cox}. Since any toric variety $X$ has a torus action over $k$ and is a twisted form of a split toric variety, the study of $X$ is equivalent to the study of the split toric variety $X_{k^s}$ with a $\Gamma$-action on the fan structure as well as the study of the open orbit $U$, see \S\ref{tv}. 

In \cite{isksurf}, Iskovskih classified minimal rational surfaces over arbitrary fields. Focusing on the cases of toric surfaces, we give an explicit description of minimal toric surfaces via toric geometry. In addition, the explicit nature of the classification of minimal toric surfaces made it possible for us to fully understand toric surfaces in aspects such as affirming Merkurjev-Panin's question (Question \ref{ques1}) in dimension $2$, decomposing toric surfaces as K-motives into products of central simple algebras, and providing full exceptional collections for the derived categories of toric surfaces, etc. 
\begin{theorem}[Theorem \ref{ms}]
The surface $X$ is a minimal smooth projective toric surface if and only if X is \hypertarget{i}{(i)} a $\P^1$-bundle over a smooth conic curve but not a form of $F_1=\text{Proj}(\O_{\P^1}\oplus\O_{\P^1}(1))$; \hypertarget{ii}{(ii)} the Severi-Brauer surface; \hypertarget{iii}{(iii)} an involution surface; \hypertarget{iv}{(iv)} the del Pezzo surface of degree $6$ with Picard rank $1$.
\end{theorem}
This paper is motivated by ideas in \cite{mp}, which studies toric varieties over an arbitrary field in the motivic category $\C$ defined in \textit{loc. cit.}, and in particular by the following question:
\begin{ques}\label{ques1}
If $X$ is a smooth projective toric variety over $k$, is $K_0(X_{k^s})$ always a permutation $\Gamma$-module?
\end{ques}
\begin{defn}
A $\Gamma$-module $M$ is a \textit{permutation $\Gamma$-module} if there exists a $\Gamma$-invariant $\Z$-basis of $M$. We call such a basis a \textit{permutation $\Gamma$-basis} or \textit{$\Gamma$-basis}.
\end{defn}
The reason that we care about the $\Gamma$-action on $K_0(X_{k^s})$ is that it in some way determines $X$, see \S\ref{sa}. For example, if $X$ has a rational point and $K_0(X_{k^s})$ is a permutation $\Gamma$-module, then $X$ is isomorphic to the \'{e}tale algebra corresponding to any $\Gamma$-basis of $K_0(X_{k^s})$ in the motivic category $\C$ \cite[Proposition 4.5]{mp}. In general, if $K_0(X_{k^s})$ has a permutation $\Gamma$-basis of line bundles over $X_{k^s}$, then the variety $X$ decomposes into a finite product of central simple algebras (over separable field extensions of $k$) in the motivic category $\C$ completely described by this $\Gamma$-basis as follows:
\begin{theorem}[Theorem \ref{basis}]\label{introbasis}
Let $X$ be a smooth projective toric $T$-variety over $k$ that splits over $l$ and $G=\Gal(l/k)$. Assume $K_0(X_l)$ has a permutation $G$-basis $P$ of line bundles on $X_l$. Let $\{P_i\}_{i=1}^t$ be $G$-orbits of $P$, and let $\pi: X_l\to X$ be the projection. For any $S_i\in P_i$, set $B_i=\End_{\O_X}(\pi_*(S_i))$ and $B=\prod_{i=1}^t B_i$. Then the map $u=\bigoplus_{i=1}^t \pi_*(S_i): X\to B$ gives an isomorphism in the motivic category $\C$.
\end{theorem}
Using the classification of minimal toric surfaces, we obtain that any smooth projective toric surface satisfies the conditions of the above theorem:
\begin{theorem}[Theorem \ref{K0-surface}]
Let $X$ be a smooth projective toric $T$-surface over $k$ that splits over $l$ and $G=\Gal(l/k)$. Then $K_0(X_l)$ has a permutation $G$-basis of line bundles on $X_l$.
\end{theorem}
The original motivation for finding the decomposition of a smooth projective variety over $k$ into a product of central simple algebras in $\C$ is to compute higher algebraic K-theory of the variety. Quillen \cite{quillen} computed higher algebraic K-theory for Severi-Brauer varieties, see Example \ref{sb}, and Swan \cite{swanqdric} for quadric hypersurfaces. Panin \cite{pflag} generalized their results by finding the decomposition in $\C$ for twisted flag varieties. 

As a matter of fact, these central simple algebras also encode arithmetic/geometric information about the variety, and in nice cases, classify its twisted forms. Blunk investigated del Pezzo surfaces of degree $6$ over $k$ in \cite{mb} in this direction, see Example \ref{dp}. He showed that a del Pezzo surface of degree $6$ is determined by a pair of Azumaya algebras (over \'{e}tale quadratic and cubic extensions of the base field, respectively) and the surface has a rational point if and only if both Azumaya algebras in the pair are split. We will investigate the same information for all smooth projective toric surfaces over $k$, see $\S$\ref{sa-surface}. For example, we obtain that a $\P^1$-bundle over a smooth conic curve is isomorphic to $k\times Q\times k\times Q$ in $\C$ and the surface is determined by the quaternion algebra $Q$ corresponding to the conic curve. More generally, if the Picard group $\Pic(X_{k^s})$ of a smooth projective toric variety $X$ is a permutation $\Gamma$-module, then the open orbit $U$ is determined by a set of central simple algebras, each corresponding to a $\Gamma$-orbit of $\Pic(X_{k^s})$, see Corollary \ref{pic}. This implies that the toric variety $X$ has a rational point if and only if every central simple algebra in the set is split. 

Moreover, since Tabuada \cite[Theorem 6.10]{tamp} showed that the motivic category $\C$ is a part of the category of noncommutative motives $Hmo_0$, it implies that certain semiorthogonal decompositions of the derived category of a smooth projective variety will give a decomposition of the variety in $\C$ (Theorem \ref{derivedtomotivic}).

We will briefly discuss the possibility of lifting the motivic decomposition of a smooth projective toric variety to the derived category, see \S\ref{dcat}. By the classification of minimal toric surfaces and known results of semiorthogonal decomposition of rational surfaces, we can confirm the lifting for smooth projective toric surfaces. 
\begin{theorem}[Theorem \ref{dersurf}]
Let $X$ be a smooth projective toric surface over $k$ that splits over $l$ and $G=\Gal(l/k)$. Then $K_0(X_l)$ has a permutation $G$-basis $P$ of line bundles over $X_l$ such that each $G$-orbit is an exceptional block. Furthermore, there exists an ordering of the $G$-orbits $\{P_i\}_{i=1}^t$ of $P$ such that $\{P_1,\dots, P_t\}$ gives a full exceptional collection of $D^b(X_l)$. Therefore, for any $S_i\in P_i$, $\{\pi_*S_1$, $\dots, \pi_*S_t\}$ is a full exceptional collection of $D^b(X)$ where $\pi: X_l\to X$ is the projection.
\end{theorem}
\subsection{Organization}
The organization of the paper is as follows: 

Sections \ref{mc} and \ref{tv} introduce the background on the motivic category $\C$ and toric varieties over $k$, including some basic facts and examples needed for the paper. For more details about $\C$, see \cite[$\S1$]{mp} or \cite[$\S3$]{mequiv}. Section \ref{mts} classifies minimal smooth projective toric surfaces over $k$ via toric geometry. Section \ref{k0s} verifies that $K_0(X_{k^s})$ has a permutation $\Gamma$-basis of line bundles for toric surfaces. In section \ref{sa}, we consider smooth projective toric varieties $X$ of all dimensions where $K_0(X_{k^s})$ has a permutation $\Gamma$-basis of line bundles. We decompose such $X$ into a product of central simple algebras in the motivic category by reinterpreting the construction of the separable algebra corresponding to a toric variety investigated in \cite{mp}. In section \ref{sa-surface}, we apply the construction in $\S$\ref{sa} to toric surfaces. Moreover, we relate the constructed algebras to the open orbit $U$ via Galois cohomology. For details on Galois cohomology, see \cite{galoiscoh}\cite{boi}\cite{csa}. In section \ref{dcat}, we discuss the relationship between the semiorthogonal decomposition of the derived category and the motivic decomposition of toric varieties via noncommutative motives and descent theory for derived categories.

Most of the time, instead of working with $X_{k^s}$ and $\Gamma$-action, we work with $X_l$ and $G=\Gal(l/k)$-action where $l$ is the splitting field of the torus $T$.
\subsection{Notation}
Fix the base field $k$ and a separable closure $k^s$ of $k$. Let $\Gamma=\Gal(k^s/k)$. Let $T$ denote an algebraic torus over $k$ with splitting field $l$ and $G=\Gal(l/k)$ unless otherwise stated. For any object $Z$ (algebraic groups, varieties, algebras, maps) over $k$ and any extension $K/k$, write $Z\otimes_k K$ as $Z_K$. 

For a split toric variety $Y$, we denote $\Sigma$ the fan structure and $\Aut_{\Sigma}$ the group of fan automorphisms. We will freely use the same notation for the ray in the fan, the minimal generator of the ray in the lattice and the Weil divisor corresponding to the ray when the context is clear.

For an algebra $A$, denote $A^{\op}$ its opposite algebra. Denote $S_n$ the permutation group of a set of $n$ elements.
\subsection{Acknowledgements}
I would like to thank my advisor, Christian Haesemeyer, for proposing this question and for his help, and to thank Alexander Merkurjev for useful conversations. I also want to thank the referees for numerous helpful comments and suggestions. Lastly, I want to thank Patrick McFaddin for pointing out a mistake in an earlier version. This paper is a version of my Ph.D. thesis at University of California, Los Angeles.
\section{The Motivic Category \texorpdfstring{$\C$}{C}}\label{mc}
\begin{defn} The \textit{motivic category} $\C=\C_k$ over a field $k$ has:
\begin{itemize}
\item Objects: Pairs $(X,A)$ where $X$ is a smooth projective variety over $k$ and $A$ is a finite separable $k$-algebra
\item Morphisms: $\Hom_{\C}((X,A),(Y,B))=K_0(X\times Y, A^{\op}\otimes_k B)$
\end{itemize}
\end{defn}
The Grothendieck group $K_0$ of a pair is defined below. A $k$-algebra $A$ is \textit{finite separable} if $\dim_k(A)$ is finite and for any field extension $K$ of $k$, the $K$-algebra $A_K$ is semisimple. Equivalently we have:
\begin{defn}\label{separable}
The algebra $A$ is a \textit{finite separable $k$-algebra} if it is a finite product of central simple $l_i$-algebras $A_i$ where $l_i$ is a finite separable field extension of $k$, i.e, $A_i$ is a matrix algebra over a finite dimensional division algebra with center $l_i$.
\end{defn}
Let $u: (X,A)\to (Y,B)$ and $v: (Y,B)\to (Z,C)$ be morphisms in $\C$. Since $u\in K_0(X\times Y, A^{\op}\otimes_k B)\cong K_0(Y\times X, B\otimes_k A^{\op})$, the map $u$ can also be viewed as $u^{\op}: (Y,B^{\op})\to (X,A^{\op})$. The \textit{composition} $v\circ u: (X,A)\to (Z,C)$ is given by  
\begin{equation*}
\pi_*(q^*v \otimes_B p^*u)
\end{equation*}
where $p: X\times Y\times Z\to X\times Y$, $q: X\times Y\times Z\to Y\times Z$, $\pi: X\times Y\times Z\to X\times Z$ are projections.

We write $X$ for $(X,k)$ and $A$ for $(\Spec k, A)$. Since the morphisms are defined in $K_0$, the category is also called the \textit{category of K-correspondences}. 
\subsection{Algebraic K-theory of a pair}
The algebraic K-theory of a pair $(X,A)$ is defined in the following way and it generalizes the Quillen K-theory of varieties:

Let $\mathcal{P}(X,A)$ be the exact category of left $\O_X\otimes_k A$-modules which are locally free $\O_X$-modules of finite rank and morphisms of $\O_X\otimes_k A$-modules. The group $K_n(X,A)$ of the pair $(X,A)$ is defined as $K_n^Q(\mathcal{P}(X,A))$, the Quillen $K$-theory of $\mathcal{P}$. Let $\mathcal{M}(X,A)$ be the exact category of left $\O_X\otimes_k A$-modules which are coherent $\O_X$-modules and morphisms of $\O_X\otimes_k A$-modules. The group $K_n'(X,A)$ of the pair $(X,A)$ is defined as $K_n^Q(\mathcal{M}(X,A))$. The embedding $\mathcal{P}\subset\mathcal{M}$ induces a map $K_n(X,A)\to K_n'(X,A)$ and it is an isomorphism if $X$ is regular (resolution theorem). Note that $K_n(X,k)$ is the usual $K_n(X)$ and $K_n(\Spec k, A)=K_n(\text{Rep}(A))$ is the $K$-theory of representations of $A$.

In fact, $K_n$ defines a functor $K_n: \C\to\textit{Ab}$ which sends $(X,A)$ to $K_n(X,A)$. For $u:(X,A)\to(Y,B)$, $x\in K_n(X,A)$, we can define 
\begin{equation*}
K_n(u)(x)=q_*(u\otimes_A p^*x)
\end{equation*}
where $p:X\times Y\to X, q:X\times Y\to Y$ are projections. 

Similarly we can define, for any variety $V$ over $k$, a functor $K_n^V:\C\to\textit{Ab}$ where on objects $K_n^V(X,A)=K_n'(V\times X,A)$.
\begin{ex}{\cite[Example 1.6(1)]{mp}}
$\M_n(k)\cong k$ in $\C$.
\end{ex}
\begin{ex}{\cite[Example 1.6(3)]{mp}, see also \cite[Theorem 9.1]{tamp}.}
Let $A$ and $B$ be two central simple $k$-algebras. Then $A\cong B$ in $\C$ if and only if $[A]=[B]\in\Br(k)$.
\end{ex}
\begin{proof}
The previous example indicates that Brauer equivalences give isomorphisms in $\C$. So $[A]=[B]\in\Br(k)$ implies $A\cong B$ in $\C$. 

For the opposite direction, since each central simple $k$-algebra is Brauer equivalent to a unique division $k$-algebra, we can assume $A, B$ are division algebras. Let $M: A\to B$ and $N: B\to A$ be inverse maps in $\C$. Since $K_0(A^{\op}\otimes_k B)\cong\Z R$ and $K_0(B^{\op}\otimes_k A)\cong\Z R^{\op}$ for $R$ the unique simple $B$-$A$-bimodule, we have $M=nR$ and $N=mR^{\op}$ for some $m,n\in\Z$. $N\circ M=N\otimes_B M\cong mn\, R^{\op}\otimes_B R\cong A$, $M\circ N=M\otimes_A N\cong mn\, R\otimes_A R^{\op}\cong B$. Since $A,B$ are simple modules, we have $mn=1$ and we can assume $M=R, N=R^{\op}$. As a right $A$-module and a left $B$-module respectively, we have $M_A\cong A^r$ and $_B M\cong B^s$. Similarly, $_A N\cong A^p$ and $N_B\cong B^q$. The left $A$-module isomorphism $N\otimes_B M\cong N\otimes_B B^s\cong N^s\cong A^{ps}\cong A$ implies that $p=s=1$. Similarly $r=q=1$. In particular, this implies $\dim_k A=\dim_k B$. 

Finally consider the $k$-algebra homomorphism $f: B\to \End_A(M_A)\cong A$ by sending $b$ to $l_b$ left multiplication by $b$. This is obviously injective, and it is surjective because $A, B$ have the same dimension, so $A\cong B$ as $k$-algebras.
\end{proof}
\section{Toric Varieties}\label{tv}
Let $T$ be an algebraic torus over $k$.
\begin{defn}
A \textit{toric $T$-variety} $X$ over $k$ is a normal geometrically irreducible variety with an action of the torus $T$ and an open orbit $U$ which is a principal homogeneous space over $T$.
\end{defn}
By definition, the torus $T_{k^s}\cong \G_{m, k^s}^n$ splits where $n=\dim X$. The torus $T$ corresponds to a cocycle class $[\rho]\in H^1(\Gamma, \Aut_{\text{gp}, k^s}(\G_{m, k^s}^n))= H^1(\Gamma, \GL(n,\Z))$ where $\Aut_{\text{gp}, k^s}$ denotes the group automorphism over $k^s$. Moreover, the torus $T$ splits over a finite Galois extension $l$ of $k$ ($T_l\cong \G_{m,l}^n$), which is called the \textit{splitting field} of $T$.

Explicitly, tori $T_{k^s}=T\otimes_k k^s$ and $\G_{m, k^s}^n=\G_{m,k}\otimes_k k^s$ have natural Galois actions with $\Gamma$ acting on the factor $k^s$. The Galois actions give group automorphisms of $T_{k^s}$ and $\G_{m, k^s}^n$ over $k$, but not over $k^s$ because $\Gamma$ also acts on the scalars $k^s$. Let $\sigma: \Gamma\to \Aut_k(T_{k^s})$ and $\tau: \Gamma\to \Aut_k(\G_{m, k^s}^n)$ be the respective natural Galois actions. Let $\phi: T_{k^s}\to\G_{m, k^s}^n$ be an isomorphism. Then we obtain $\rho:\Gamma\to\GL(n,\Z)$ by sending $g$ to $\phi \sigma(g) \phi^{-1}\tau(g)^{-1}$, and we have $\ker(\rho)=\Gal(k^s/l)$ where $l$ is the splitting field. 

Conversely, the torus $T$ can be constructed from $\rho: \Gamma\to\GL(n,\Z)$ as follows, see \cite[\S1]{vdemazure}. The map $\rho$ factors through $\rho': G=\Gal(l/k)\to\GL(n,\Z)$ for a finite Galois extension $l$ of $k$. Let $\mu: G\to\Aut_k(\G_{m,l}^n)$ be the action on the torus $\G_{m,k}^n\otimes_k l$ via $\mu(g)=\rho'(g)\otimes g, g\in G$. Then $T\cong \G_{m,l}^n/\mu(G)$.
\begin{defn}
A toric $T$-variety $X$ over $k$ is called a \textit{toric $T$-model} if $U(k)$ is nonempty. 
\end{defn}
In this case, the open orbit $U\cong T$ as $k$-varieties and there is an $T$-equivariant embedding $T\hookrightarrow X$. If X is smooth over $k$, then the set $X(k)$ is nonempty if and only if $U(k)$ is \cite[\S4 Proposition 4]{vkrootsys}. 
\begin{defn}
A toric $T$-variety is \textit{split} if $T$ splits, and is \textit{non-split} otherwise.
\end{defn}
Let $X_{k^s}$ (or $X_l$) be the split toric variety with the fan structure $\Sigma$. Since the $\Gamma$-action on $T_{k^s}$ is compatible with the one on $X_{k^s}$, the image of $\rho$ is contained in $\Aut_{\Sigma}$, namely
\begin{equation*}
\rho(\Gamma)=\Gal(l/k)\subseteq\Aut_{\Sigma}\subset \GL(n,\Z).
\end{equation*}

Let $X_{\Sigma}$ be the split toric variety over $k$ with the fan structure $\Sigma$. If $X$ is a toric $T$-model, then similarly to the case of torus $T$, the variety $X$ can be recovered from $\rho$ and $\Sigma$ as $(X_{\Sigma}\otimes_k l)/\mu(G)$. In general, for each toric $T$-variety $X$, there is a unique (up to $T$-isomorphism) toric $T$-model $X^*$ such that $X_{k^s}\cong (X^*)_{k^s}$. We call $X^*$ the associated toric $T$-model of $X$. More specifically, the toric $T$-model $X^*$ is given by $(X\times U)/T$ where $T$ acts on $X\times U$ diagonally, and the toric $T$-variety $X$ is given by $(X^*\times U)/T$ where $T$ acts on $X^*\times U$ via $t\cdot(x,y)=(tx,yt^{-1})$, see \cite[\S4]{vkrootsys}. 

In summary, an algebraic torus $T$ is uniquely determined by a $1$-cocycle (class) $\rho: \Gamma\to\GL(n,\Z)$. A toric $T$-model $X$ is uniquely determined by $\rho$ and fan $\Sigma$ with the restriction $\rho(\Gamma)\subseteq\Aut_{\Sigma}$. A toric $T$-variety is uniquely determined by its associated $T$-model $X^*$ and a principal homogeneous space $U\in H^1(k,T)$.
\begin{lem}\label{toricmorphism}
Let $\phi: X_{\Sigma_1}\to X_{\Sigma_2}$ be a toric morphism of split smooth projective toric varieties over $k^s$, and let $\bar{\phi}: N_1\to N_2$ be the induced $\Z$-linear map of lattices that is compatible with fans $\Sigma_1, \Sigma_2$. Let $\rho_i: \Gamma\to \Aut(N_i)$ be Galois actions on $N_i$ that are compatible with the fans $\Sigma_i$ ($\rho_i(\Gamma)\subseteq \Aut_{\Sigma_i}$) such that $\bar{\phi}$ is $\Gamma$-equivariant with respect to $\rho_1, \rho_2$. Let $T_i$ be the torus corresponding to $\rho_i$. Then, for any $U_1\in H^1(k,T_1)$, there exists $U_2\in H^1(k, T_2)$ such that $\phi$ descends to a map $X_1\to X_2$ where $X_i$  is the toric variety corresponding to $(\rho_i,\Sigma_i, U_i)$ for $i=1,2$.
\end{lem}
\begin{proof}
Restrict $\phi$ to tori $\phi |_{T_{N_1}}: T_{N_1}\to T_{N_2}$. Since $\bar{\phi}$ is $\Gamma$-equivariant, the maps $\phi$ and $\phi |_{T_{N_1}}$ descend to $\varphi: X_1^*\to X_2^*$ where $X_i^*$ are the toric $T_i$-models corresponding to $\Sigma_i$ and $\psi: T_1\to T_2$. The map $\psi$ induces $H^1(k, T_1)\to H^1(k, T_2)$ and let $U_2$ be the image of $U_1$ under this map. Set $X_i=(X_i^*\times U_i)/T_i$. Then $\phi$ descends to a map $X_1\to X_2$.
\end{proof}
\begin{ex}\label{sb}
Severi-Brauer variety $X$ ($X_{k^s}\cong\P^n$) . Let $A$ be a central simple $k$-algebra of degree $n+1$. Then $X=\text{SB}(A)$ is a toric variety with the torus $T=\text{R}_{E/k}(\G_{m,E})/\G_{m,k}$ where $E$ is a maximal \'{e}tale $k$-subalgebra of $A$. The variety $X$ has a rational point if and only if $A=M_{n+1}(k)$ if and only if $X\cong \P^n$.

Quillen \cite[$\S 8$ Theorem $4.1$]{quillen} showed that $K_m(\text{SB}(A))\cong K_m(k)\times \prod_{i=1}^n K_m(A^{\otimes i})$ for $m\geqslant 0$, and Panin \cite{pflag} showed that $\text{SB}(A)\cong k\times \prod_{i=1}^n A^{\otimes i}$ in $\C$.
\end{ex}
\begin{ex}\label{dp}
Let $X$ be a del Pezzo surface of degree $6$ over $k$ ($K_X$ is anti-ample with $K_X^2=6$, $X_{k^s} \cong \Bl_{p_1,p_2,p_3}(\P^2)$ where $p_1,p_2,p_3$ are not collinear). It is a toric $T$-variety where the torus $T$ is the connected component of the identity of $\Aut_k(X)$. 

Blunk \cite{mb} showed that $X\cong k\times P\times Q$ in $\C$ where $P$ is an Azumaya $K$-algebra of rank $9$ ($\dim_k(P)/\dim_k(K)=9$) and $Q$ is an Azumaya $L$-algebra of rank $4$ where $K,L$ are \'{e}tale $k$-algebras of degree $2$ and $3$, respectively.
\end{ex}
\begin{ex}\label{inv}
Involution surface $X$ ($X_{k^s}\cong \P^1\times\P^1$). The surface $X$ corresponds to a central simple $k$-algebra $A$ of degree 4 together with a quadratic pair $(\sigma, f)$ on $A$. For the definition of a quadratic pair, see \cite[\S 5B]{boi}. The associated even Clifford algebra $C_0(A, \sigma, f)$ (defined in their \S8B) is a quaternion algebra over $K$, which is an \'{e}tale quadratic extension of $k$ and is called the \textit{discriminant extension} of $X$. Write $B=C_0(A, \sigma, f)$. Then $X$ is the Weil restriction $\text{R}_{K/k}\text{SB}(B)$, see \cite[Example 3.3]{aberdel}. Denote by $T$ the torus of $\text{SB}(B)$ in Example \ref{sb}. Then $X$ is a toric variety with the torus $\text{R}_{K/k}T$.

Panin \cite{pflag} showed that $X\cong k\times B\times A$ in $\C$.
\end{ex}
\subsection{\texorpdfstring{$K_0$}{K0} of split toric varieties}
Let $Y$ be a split smooth proper toric $T$-variety with fan $\Sigma$. 

For $\sigma\in\Sigma$, denote $\O_{\sigma}$ the closure of the $T$-orbit corresponding to $\sigma$ and $J_{\sigma}$ the sheaf of ideals defining $\O_{\sigma}$. Write $\sigma(1)$ as the set of rays spanning $\sigma$. For $\sigma, \tau\in\Sigma$, if $\sigma(1)\cap\tau(1)=\emptyset$ and $\sigma(1)\cup\tau(1)$ span a cone in $\Sigma$, then denote the cone by $\<\sigma,\tau\>$, otherwise set $\<\sigma,\tau\>=0$.

From \cite{klydemazure}, we have
\begin{thm}[Klyachko, Demazure]\label{K0}
As an abelian group, $K_0(Y)$ is generated by $\O_{\sigma}=1-J_{\sigma}$ with these relations:
\begin{equation}\label{rel1}
    \O_{\sigma}\cdot \O_{\tau}=\left\{
    \begin{array}{ll}
        \O_{\<\sigma,\tau\>}, & \<\sigma,\tau\>\neq 0\\
         0, & \text{otherwise}
    \end{array}
    \right.
\end{equation}
\begin{equation}\label{rel2}
    \prod_{e\in\Sigma(1)} J_e^{f(e)}=1, f\in\Hom(N,\Z)=M \text{(the group of characters of $T$)}.
\end{equation}
\end{thm}
\begin{thm}[Klyachko] 
The abelian group $K_0(Y)$ is free with rank equal to the number of the maximal cones. In addition, sheaves $\O_y$ and $\O_{y'}$ coincide in $K_0(Y)$ for any rational closed points $y, y'\in Y$.
\end{thm}
\section{Minimal Toric Surfaces}\label{mts}
Let $X$ be a smooth projective toric surface over $k$. We say $X$ is \textit{minimal} if any birational morphism $f:X\to X'$ from $X$ to another smooth surface $X'$ defined over $k$ is an isomorphism. In this section, we will classify minimal smooth projective toric surfaces.

First we notice that the exceptional locus of any birational morphism from a toric surface is torus invariant. We use the convention that a surface is integral, separated and of finite type.
\begin{lem}
Let $W$ be a smooth projective toric $T$-surface over $k$. Let $h: W\to Z$ be a birational morphism over $k$ from $W$ to a smooth surface $Z$ over $k$. Let $E$ be the exceptional divisor of $h$. Then $E$ is $T$-invariant. Therefore, the surface $Z$ is a smooth projective toric $T$-surface and the map $h$ is $T$-invariant.
\end{lem}
\begin{proof}
First assume that $k$ is separably closed. Then $W$ is split. Since for a split toric variety the group of $T$-invariant Cartier divisors $\CDiv_T$ maps onto the Picard group, the line bundle $\O(E)$ is fixed by the $T$-action. For any $t\in T$, the divisor $tE$ is linearly equivalent to $E$ (denoted $tE\sim E$).

Now assume the locus $E$ is not $T$-invariant and let $t_0\in T$ be such that $t_0E\neq E$. Note that since $W$ is proper and $Z$ is separated, the map $h$ is proper and the surface $Z=h(W)$ is also proper (thus projective). We have $p(t_0E)\sim p(E)=0$. Let $C=p(t_0E)$ which is a curve on $Z$. Embedding $Z$ into some $\mathbb{P}^n$ and let $H$ be a hyperplane of $\mathbb{P}^n$. Since $C$ is a curve, we have $C.H>0$. Therefore, $C$ cannot be linearly equivalent to $0$, a contradiction.

For an arbitrary field $k$, we base change to the separable closure $k^s$ and use the same argument.
\end{proof}
\begin{lem}\label{minsurf}
Let $X$ be a smooth projective toric $T$-surface over $k$. Then $X$ is minimal if and only if $X_{k^s}$ admits no $\Gamma$-invariant set of pairwise disjoint $T_{k^s}$-invariant $(-1)$-curves.  
\end{lem}
\begin{proof}
Since any $(-1)$-curve is the exceptional locus of some birational morphism, by the previous lemma, it is always torus invariant. The rest follows from \cite[Theorem 3.2]{hassurf}.
\end{proof}
\begin{defn}
Let $Y$ be a split smooth projective toric surface over a field $K$. If there is a finite group $G$ acting on $Y$ by $K$-automorphisms, we call $Y$ a \textit{$G$-surface} over $K$.  The $G$-surface $Y$ is called \textit{$G$-minimal} over $K$ if $Y$ admits no $G$-invariant set of pairwise disjoint torus invariant $(-1)$-curves.
\end{defn}
Lemma \ref{minsurf} implies that we can redefine minimal toric surfaces as follows:
\begin{defn}
Let $X$ be a smooth projective toric $T$-surface over $k$ and let $\rho:\Gamma\to\GL(2,\Z)$ be the map corresponding to the torus $T$. Let $G=\rho(\Gamma)$ which is a finite subgroup of $\GL(2,\Z)$ and acts on the split toric surface $X_{k^s}$ by fan automorphisms ($G\subseteq\Aut_{\Sigma}(X_{k^s})$). We say the toric surface $X$ is \textit{minimal} if $X_{k^s}$ is $G$-minimal over $k^s$.
\end{defn}
\begin{prop}
Let $X$ and $G=\rho(\Gamma)$ be the same as above. Then there is a finite chain of blow-ups of toric $T$-surfaces
\begin{equation*}
X=X_0\overset{f_1}{\rightarrow} X_1\overset{f_2}{\to}\cdots\overset{f_n}{\to} X_n=X'
\end{equation*}
where each $X_i$ is a smooth projective toric $T$-surface, each map $f_i$ is the blow-up of $X_i$ along $T$-invariant reduced zero-dimensional subscheme (in particular, $f_i$ is $T$-invariant) and $X'$ is minimal.
\end{prop}
\begin{proof}
If $X$ is not minimal, then $X_{k^s}$ admits a $G$-invariant set of pairwise disjoint $T_{k^s}$-invariant $(-1)$-curves. Contracting this $G$-set of $(-1)$-curves and descending the contraction map to the base field $k$, we get a map $f_1: X\to X_1$ which is the blow-up of a smooth projective toric $T$-surface $X_1$ along $T$-invariant reduced zero-dimensional subscheme. This process will terminate in finite steps because the number of rays in the fan of $(X_1)_{k^s}$ is strictly less than that of $X_{k^s}$.
\end{proof}
Now, classifying all minimal smooth projective toric surfaces over $k$ is the same as classifying, for each finite subgroup $G$ of $\GL(2,\Z)$ (up to conjugacy), $G$-minimal toric surfaces over $k^s$. It is well known that, when $G$ is trivial, the minimal (toric) surfaces are $\P^2$ and Hirzebruch surfaces $F_a=\text{Proj}(\O_{\P^1}\oplus\O_{\P^1}(a))$ for $a\geqslant 0, a\neq 1$.

There are $13$ non-conjugate classes of finite subgroups of $\GL(2,\Z)$ and they can only be either cyclic or dihedral groups \cite[Chapter IX, \S 14]{intmatrices}. See Table \ref{table:subgroups}.

\begin{table}[!htb]
\caption{Non-conjugate classes of finite subgroups of $\GL(2,\Z)$ and their generators}
\label{table:subgroups}
\begin{center}
\begin{tabular}{lll}
\hline
Cyclic \hspace{15ex} & Dihedral\hspace{15ex} & Generators\\
\hline
$C_1=\<I\>$   & $D_2=\<C\>$  & \multirow{2}{*}{$A =\left(\begin{array}{rr} 1 & -1 \\ 1 & 0 \end{array}\right)$}\\
              & $D_2'=\<C'\>$ & \\[1.5ex]
$C_2=\<-I\>$  & $D_4=\<-I,C\>$ & \multirow{2}{*}{$B = \left(\begin{array}{rr} 0 & -1 \\ 1 & 0 \end{array}\right)$}\\
              & $D_4'=\<-I,C'\>$ & \\[1.5ex]
$C_3=\<A^2\>$ & $D_6=\<A^2,C\>$ & \multirow{2}{*}{$C = \left(\begin{array}{rr} 0 & 1 \\ 1 & 0 \end{array}\right) $}\\
              & $D_6'=\<A^2,-C\>$ & \\[1.5ex]
$C_4=\<B\>$   & $D_8=\<B,C\>$  &  \multirow{2}{*}{$C' = \left(\begin{array}{rr} 1 & 0 \\ 0 & -1 \end{array}\right) $}\\
$C_6=\<A\>$   & $D_{12}=\<A,C\>$ & \\
\hline
\end{tabular}
\end{center}
\end{table}
\begin{defn}
Let $Y$ be a split smooth projective toric surface with fan structure $\Sigma$. Counterclockwise label the rays of $\Sigma$ as $y_1,...,y_n$ and denote by $D_i$ the divisor corresponding to $y_i$. We can assign a sequence $a=(a_1,...,a_n)$ to $Y$ where $a_i=D_i^2$. We refer to this sequence as the \textit{self-intersection sequence} of $Y$. 
\end{defn}
The group of fan automorphisms $\Aut_{\Sigma}(Y)$ acts on $\Z^2$, permuting rays $y_i$ of the fan $\Sigma$. First observe that as automorphisms of $Y$, the group $\Aut_{\Sigma}(Y)$ preserves the self-intersection number of any divisor and thus permutes (torus invariant) $(-1)$-curves on $Y$. Now, let us consider the case where $\Aut_{\Sigma}(Y)\cap \SL(2,\Z)=C_t$ is nontrivial and look at the action of $C_t$ on the rays. As indicated in Table \ref{table:subgroups}, the cyclic group $C_t$ is generated by powers of $A$ or $B$ where $B$ is the rotation by $\pi/4$ and $A$ is conjugate in $\GL(2,\mathbb{R})$ to the rotation by $\pi/3$. In particular, the action of $C_t$ on the fan $\Sigma$ is free which implies $t\,|\,n$.
\begin{lem}\label{blow down}
Let $\Aut_{\Sigma}(Y)\cap \SL(2,\Z)=C_t$ be nontrivial (i.e, $t=2,3,4,6$). If the number of rays of the fan $>\max\{4,t\}$, then $Y$ is not $C_t$-minimal, that is, there exists a $C_t$-invariant set of pairwise disjoint $(-1)$-curves on $Y$. Therefore, $C_t$-minimal surfaces have the number of rays $\leqslant\max\{4,t\}$.
\end{lem}
\begin{proof}
Denote counterclockwise $y_1,...,y_n$ as rays of $\Sigma$ and let $a=(a_1,...,a_n)$ be its self-intersection sequence. If $n>4$, $Y$ is not $\P^2$ or $F_a$, then there exists $i$ such that $a_i=-1$. Let $\sigma$ be a generator of $C_t$ and as discussed above, $\sigma$ rotates the rays. If $n>t$, then the ray $\sigma(y_i)$ is not adjacent to $y_i$ (i.e, corresponding divisors are disjoint) and thus $\{y_i,\sigma(y_i),\dots,\sigma^{t-1}(y_i)\}$ form a $C_t$-invariant set of pairwise disjoint $(-1)$-curves.
\end{proof}
\begin{lem}
$D_2$ fixes rays generated by $\pm (1,1)$ or maximal cones generated by $(1,0)$ and $(0,1)$ or by $(-1,0)$ and $(0,-1)$; $D_2'$ fixes rays generated by $\pm (1,0)$.
\end{lem}
Using toric geometry, Oda showed in \cite[Theorem 8.2]{oda} that a split smooth projective toric surface is a succession of blow-ups of $\P^2$ or $F_a$.  The proof of the theorem is essentially the following lemma:
\begin{lem}\label{odalem}
Let $Y$ be a split smooth projective toric surface with the fan $\Sigma$. Let $x,y$ be two rays in $\Sigma$ where their minimal generators form a basis of $\Z^2$. If $x,y$ are not adjacent in the fan, then there is a ray $z\in\Sigma$ between $x,y$ corresponding to a $(-1)$-curve.
\end{lem}
Now we are ready to classify $G$-minimal toric surfaces for $G$ a finite subgroup of $\GL(2,\Z)$.
\begin{prop}\label{surface}
Let $Y$ be a split smooth projective toric surface and let $G$ be a finite subgroup of $\GL(2,\Z)$ acting on $Y$ by fan automorphisms; that is, $G\subseteq\Aut_{\Sigma}(Y)$. Then the surface $Y$ is $G$-minimal if and only if $Y$ belongs to one of the following:
\begin{itemize}
\item $G=D_2$: $Y=\mathbb{P}^2, \P^1\times \P^1, F_{2a+1}, a\geqslant 1$;
\item $G=D_2'$: $Y=F_{2a},a\geqslant 0$;
\item $G=C_2, C_4, D_4, D_4', D_8$: $Y=\mathbb{P}^1\times\mathbb{P}^1$;
\item $G=C_3, D_6$: $Y=\mathbb{P}^2$;
\item $G=C_6, D_6', D_{12}$: $Y=S$
\end{itemize}
where $F_a=\text{Proj}(\O_{\P^1}\oplus\O_{\P^1}(a))$ is the Hirzebruch surface and $S$ is the blow-up $\Bl_{p_1,p_2,p_3}(\P^2)$ of $\P^2$ along three torus invariant points.
\end{prop}
\begin{proof}
Assume the split toric surface $Y$ is $G$-minimal. Let $\Sigma$ be the fan structure of $Y$ and let $n$ be the number of rays of $\Sigma$. It is clear that for any subgroup $H$ of $G$ together with the restricted $H$-action on $Y$, the surface $Y$ is either $H$-minimal or the (successive) blow-ups of $H$-minimal toric surfaces.
\begin{description}
\item[$G=D_2$] (I) If $D_2$ fixes at least one maximal cone, then $\Sigma$ contains (I.1) rays $(1,0)$, $(0,1)$, $(-1,-1)$ where $D_2$ fixes the maximal cone generated by $(1,0)$, $(0,1)$ or (I.2) rays $(1,0)$, $(0,1)$, $(-1,0)$, $(0,-1)$ where $D_2$ fixes the maximal cones generated by $(1,0)$, $(0,1)$ and by $(-1,0)$, $(0,-1)$. (II) Otherwise $\Sigma$ contains rays $\pm (1,1)$, and the rays counterclockwise before and after $(1,1)$ must be $(a+1,a)$ and $(a,a+1)$, respectively. By Lemma \ref{odalem}, it is easy to see that if $\Sigma$ contains more rays in any of the above cases, then $Y$ admits a $D_2$-set of pairwise disjoint $(-1)$-curves.  Thus, $Y$ is isomorphic to (I.1) $\P^2$; (I.2) $\P^1\times \P^1$; (II) $F_{2a+1}$. Since $F_1$ has a $D_2$-invariant $(-1)$-curve, it is not minimal. So we have $a\geqslant 1$.
\item[$G=D_2'$] $\Sigma$ contains rays $\pm (1,0)$, and the rays counterclockwise before and after $(1,0)$ must be $(a,-1)$ and $(a,1)$, respectively. By Lemma \ref{odalem}, $\Sigma$ contains no other rays. Thus, $Y$ is isomorphic to $F_{2a},a\geqslant 0$.
\item[$G=C_2$] Let $x,y\in\Sigma$ be two adjacent rays. Then $\Sigma$ should have rays $x,y,-x,-y$ where the minimal generators of $x,y$ form a basis of $\Z^2$ and by Lemma \ref{odalem}, it contains no other rays. Thus, $Y\cong\P^1\times\P^1$.
\item[$G=C_4, D_4, D_4', D_8$] Since $C_2$ is a subgroup of $C_4, D_4, D_4', D_8$, we have $Y\cong\P^1\times\P^1$ or its blow-ups. Since the group of fan automorphisms of $\P^1\times\P^1$ is $D_8$ which contains $C_4, D_4, D_4'$, the minimal $C_2$-surface $\P^1\times\P^1$ is already a $G$-surface for $G= C_4, D_4, D_4', D_8$ and must be $G$-minimal. Thus, $Y\cong\P^1\times\P^1$.
\end{description}

For cases $G=C_t,t>2$. Recall that $t\,|\,n$ and by Lemma \ref{blow down}, $n\leqslant\max\{4,t\}$.

\begin{description}
\item[$G=C_3$] $3\,|\,n, n\leqslant 4$, so $n=3$ and $Y\cong\P^2$.
\item[$G=D_6$] $C_3\subset D_6$ implies that $Y$ is either $\P^2$ or its blow-ups. Since the group of fan automorphisms is $D_6$, we have $Y\cong\P^2$. 
\end{description}

For cases $G\supseteq C_3$, observe that if $Y$ is not $\P^2$, then it must be the blow-up of $S$ where $S$ is the blow-up of $\P^2$ along three torus invariant points.

\begin{description}
\item[$G=C_6, D_6', D_{12}$] $C_3\subset D_6'\subset D_{12}$ and $C_3\subset C_6\subset D_{12}$ imply that $Y$ is either $\P^2$ or the blow-up of $\P^2$. Since the group of fan automorphisms of $\P^2$ is $D_6$, $Y$ can not be $\P^2$. Thus, $Y$ is either $S$ or its blow-up. We have $Y\cong S$ because the group of fan automorphisms of $S$ is $D_{12}$.
\end{description}
\end{proof}
\begin{lem}\label{cb}
Let $X$ be a toric surface that is a form of $F_a, a\geqslant 1$. Then $X$ is a $\P^1$-bundle over a smooth conic curve. If $X$ has a rational point, then $X\cong F_a$.
\end{lem}
\begin{proof}
Let $X$ correspond to $(\rho_1,\Sigma_1,U_1)$ and let $\Sigma_1$ be the fan of $F_a$ with rays $(1,0)$, $(0,1)$, $(-1,a)$, $(0,-1)$. Let $\bar{\phi}: \Z^2\to\Z$ be the projection to the first factor, which corresponds to $\phi: F_a\to \P^1$. Let $\rho_2=\det\circ \rho_1: \Gamma\to\GL(1,\Z)$. Either $\rho_1$ is trivial or $\rho_1$ permutes the rays $(1,0), (-1,a)$. Then $\bar{\phi}$ is Galois equivariant with respect to $\rho_1$ and $\rho_2$. By Lemma \ref{toricmorphism}, the map $\phi$ descends to $\varphi: X\to C$.  As a form of $\P^1$, $C$ is a smooth plane conic curve (\cite[Corollary 5.4.8]{csa} for characteristic not $2$ and \cite[\S 45A]{ekmqf} for any characteristic).

Let $D$ be the divisor corresponding to the ray $(0,-1)$. Then $D$ is a Galois invariant section of the bundle $\phi: F_a\to \P^1$. Thus, $D$ descends to a section $D'$ of $\varphi: X\to C$. Moreover, $F_a\cong\P(\phi_* \O_{F_a}(D))$ descends to $X\cong\P(\varphi_*\O_X(D'))$. Thus, $X$ is a $\P^1$-bundle over $C$. If $X$ has a rational point, so does $C$. Therefore, $C\cong\P^1$ and $X\cong F_a$.
\end{proof}
By Proposition \ref{surface}, a minimal smooth projective toric surface $X$ is a form of (i) $F_a, a\geqslant 2$; (ii) $\P^2$; (iii) $\P^1\times \P^1$; (iv) $\Bl_{p_1, p_2,p_3} (\P^2)$ where $p_1,p_2,p_3$ are not collinear. Furthermore, we have
\begin{thm}\label{ms}
The surface $X$ is a minimal smooth projective toric surface if and only if X is \hypertarget{i}{(i)} a $\P^1$-bundle over a smooth conic curve but not a form of $F_1=\text{Proj}(\O_{\P^1}\oplus\O_{\P^1}(1))$; \hypertarget{ii}{(ii)} the Severi-Brauer surface; \hypertarget{iii}{(iii)} an involution surface; \hypertarget{iv}{(iv)} the del Pezzo surface of degree $6$ with Picard rank $1$.
\end{thm}
\begin{proof}
It follows from Lemma \ref{cb}, Example \ref{sb}, \ref{dp}, \ref{inv} and the fact that a minimal del Pezzo surface of degree not equal to $8$ has Picard rank $1$ \cite[Theorem 2.4]{ctkarm}.
\end{proof}
\section{\texorpdfstring{$K_0$}{K0} of Toric Surfaces}\label{k0s}
In this section, we will show that $K_0(X_{k^s})$ is a permutation $\Gamma$-module for $X$ a smooth projective toric surface over $k$.  First recall how $K_0$ behaves under blow-ups:

\begin{thm}\cite[VII 3.7]{sga6}
Let $X$ be a noetherian scheme and $i:Y\to X$ a regular closed immersion of pure codimension $d$. Let $p:X'\to X$ be the blow up of $X$ along $Y$ and $Y'=p^{-1}Y$. There is a split short exact sequence
\begin{equation*}
0\to K_0(Y)\overset{u}{\to}K_0(Y')\oplus K_0(X)\overset{v}{\to}K_0(X')\to 0
\end{equation*}
and the splitting map $w$ for $u$ is given by $w(y',x)=p|_{Y' *}(y'), y'\in K(Y'), x\in K(X)$.
\end{thm}

This gives us an isomorphism $K_0(X')\cong \ker(w)\cong K_0(X)\oplus \bigoplus^{d-1} K_0(Y)$ which fits into the split short exact sequence
\begin{equation*}
0\to K_0(X)\overset{p^*}{\to}K_0(X')\to \bigoplus^{d-1} K_0(Y)\to 0.
\end{equation*}
Now let $X$ be a smooth projective toric $T$-surface over $k$ that splits over $l$. Let $Y$ be a $T$-invariant reduced zero-dimensional subscheme of $X$. Then $Y_l$ is a disjoint union of $T_l$-invariant points permuted by $G=\Gal(l/k)$. Set $X'=\text{Bl}_Y X$. We have
\begin{equation*}
0\to K_0(X_l)\overset{p^*}{\to}K_0(X'_l)\to K_0(Y_l)=\bigoplus\Z\to 0
\end{equation*}
where $p^*$ is a $G$-homomorphism. Each $\Z$ is generated by $\O_{E_i}(-1)$  where $E_i$ are the exceptional divisors corresponding to the points in $Y_l$ and $G$ permutes $E_i$ the same way as $G$ permutes the points in $Y_l$.

Note that $\O_{E_i}(-1)=\O_{X_l'}(E_i)-\O_{X_l'}$ in $K_0$. If we know $K_0(X_l)$ has a permutation $G$-basis $\gamma$, then $K(X'_l)$ has a permutation $G$-basis consisting of $p^*\gamma$ (total transforms of $\gamma$) and the $\O(E_i)$.

\begin{thm}\label{K0-surface}
Let $X$ be a smooth projective toric $T$-surface over $k$ that splits over $l$ and $G=\Gal(l/k)$. Then $K_0(X_l)$ has a permutation $G$-basis of line bundles on $X_l$.
\end{thm}
\begin{proof}
By previous discussion and the fact that $G\subseteq \Aut_{\Sigma}$, it suffices to prove that $K_0(X_l)$ has a permutation $\Aut_{\Sigma}$-basis of line bundles for $X$ minimal. By Theorem \ref{ms}, we only need to consider the following cases for $X_l$:
\begin{description}
\item[\hyperlink{i}{(i)}] $F_a, a\geqslant 2$, $\Aut_{\Sigma}=S_2$.
\item[\hyperlink{ii}{(ii)}] $\P^2$, $\Aut_{\Sigma}=D_6$.
\item[\hyperlink{iii}{(iii)}] $\P^1\times\P^1$, $\Aut_{\Sigma}=D_8$.
\item[\hyperlink{iv}{(iv)}] del Pezzo surface of degree $6$, $\Aut_{\Sigma}=D_{12}$.
\end{description}

We will use Equation (\ref{rel2}) in Theorem \ref{K0} with $f=(1,0)$ and $(0,1)$ in producing relations and finding a permutation basis. We will write $x_i$ for rays in the fan and $J_i=\O(-D_i)$ where $D_i$ are the divisors corresponding to $x_i$.

\textbf{\hypertarget{i'}{(i)}}: Rays $x_1=(1,0)$, $x_2=(0,1)$, $x_3=(-1,a)$, $x_4=(0,-1)$. $S_2$ fixes $x_2,x_4$ and permutes $x_1,x_3$. Relations are:
\begin{equation*}
    \left\{
    \begin{aligned}
    J_3 &=J_1\\
    J_4 &=J_2J_3^a=J_1^aJ_2
    \end{aligned}
    \right.
\end{equation*}
Let $x$ be a rational point of $X_l$. Then the sheaf $\O_x=(1-J_1)(1-J_2)$ in $K_0$.  For any $m\in\Z$, consider the exact sequence 
\begin{equation*}
0\to \O(-(m+1)D_1-D_2)\to \O(-mD_1-D_2)\to \O_{D_1}(-mD_1-D_2)\to 0. 
\end{equation*}
Since $D_1\cong\P^1$ and $\deg[\O_{D_1}(-mD_1-D_2)]=D_1\cdot(-mD_1-D_2)=-1$, we have $\O_{D_1}(-mD_1-D_2)=\O_{D_1}(-1)=\O_{D_1}-\O_x$ in $K_0$. Hence $J_1^{m+1}J_2=J_1^mJ_2+J_1J_2-J_2$ in $K_0$. This implies $J_4=J_1^aJ_2$ belongs to the abelian group generated by $1,J_1,J_2,J_1J_2$. By Theorem \ref{K0}, we have $K_0$ as an abelian group is generated by $1,J_1,J_2,J_1J_2$. They form a basis of $K_0$ because the rank of $K_0$ ($=$ the number of maximal cones in the fan) is $4$. Thus, $K_0$ has a permutation basis $1,J_1,J_2,J_1J_2$. (Alternatively, this basis can easily be obtained from the projective bundle theorem \cite[\S8, Theorem 2.1]{quillen} because $F_a$ is a $\P^1$-bundle over $\P^1$.)

\textbf{\hypertarget{ii'}{(ii)}}: Rays $x_1=(1,0)$, $x_2=(0,1)$, $x_3=(-1,-1)$. $D_6$ rotates $x_i$ and reflects along lines in $x_1,x_2,x_3$. Relations are $J_1=J_2=J_3$. A permutation basis is $1,J_1,J_1^2$.

\textbf{\hypertarget{iii'}{(iii)}}: Rays $x_1=(1,0)$, $x_2=(0,1)$, $x_3=(-1,0)$, $x_4=(0,-1)$. $D_8$ rotates $x_i$ and reflects along lines in $x_1,x_2,(1,1),(-1,1)$. Relations are:
\begin{equation*}
    \left\{
    \begin{aligned}
    J_3 &=J_1\\
    J_4 &=J_2
    \end{aligned}
    \right.
\end{equation*}
A permutation basis is $1,J_1,J_2,J_1J_2$.

\textbf{\hypertarget{iv'}{(iv)}}: Rays $x_1=(1,0)$, $x_2=(0,1)$, $x_3=(-1,-1)$, $y_1=(-1,0)$, $y_2=(0,-1)$, $y_3=(1,1)$. $D_{12}\cong S_2\times S_3$ ($S_2,S_3$ permutation groups), $S_2=\<-1\>$ switches between $x_i$ and $y_i$. $S_3$ permutes the pair of rays $(x_i,y_i)$. Let $D_i'$ be the divisors corresponding to the rays $y_i$ and let $J_i'=\O(-D_i')$. Relations are
\begin{equation*}
\frac{J_1}{J_1'}=\frac{J_2}{J_2'}=\frac{J_3}{J_3'}
\end{equation*}
As proved in \cite[Theorem 4.2]{mb}, we have a permutation basis $1,R_1,R_2,R_3,Q_1,Q_2$ where
\begin{equation*}
    \left\{
    \begin{aligned}
    R_1 &= J_1J_2'\\
    R_2 &= J_2J_3'\\
    R_3 &= J_3J_1'\\
    Q_1 &= J_1J_2J_3'\\
    Q_2 &= J_1'J_2'J_3
    \end{aligned}
    \right.
\end{equation*}
\end{proof}

\begin{rem}
The difficulties in generalizing Theorem \ref{K0-surface} to higher dimensions (at least using the approach of this paper) are: 

(1) The classification of non-conjugacy classes of finite subgroups of $\GL(n,\Z)$ is difficult and not complete. It often only provides algorithms and requires the help of a computer even for small $n$. Also, the number of those finite subgroups grows very fast relative to $n$. For example, there are total of $73$ for $\GL(3,\Z)$ and $710$ for $\GL(4,\Z)$.  

(2) The $K$-group $K_0(X_l)$ in question may not stay a permutation module after blow-ups if $X$ is not a surface.
\end{rem}
\section{Construction of Separable Algebras}\label{sa}
Let $X$ be a smooth projective toric $T$-variety over $k$ that splits over $l$, and let $X^*$ be its associated toric model, see \S\ref{tv}. \cite[Theorem 5.7]{mp} states that there is a split monomorphism $u:X^*\to A$ in the motivic category $\C$ from $X^*$ to an \'{e}tale $k$-algebra $A$ and $u$ is represented by an element $Q$ in $\Pic(X^*\otimes_k A)$.  Using the invertible sheaf $Q$, a map $u': X\to B$ can be constructed out of $u$. Theorem 7.6 of the same work states that $u'$ is also a split monomorphism in $\C$. In this section, we will recall the construction of $u'$ and consider the case when $u$ is an isomorphism.

Write $X_A=X\otimes_k A$ and we have $f: X_l\to X^*_l$, a $T_l$-isomorphism. Consider the diagram:
\begin{equation}\label{endo}
\begin{tikzcd}
X_{A\otimes_k l} \arrow{r}{f_A} \arrow{d}{\pi_{X_A}} & X^*_{A\otimes_k l} \arrow{d}{\pi_{X^*_A}}\\
X_A  & X^*_A
\end{tikzcd}
\end{equation}

Let $P'=f^*(\pi^*_{X^*_A}(Q))$. Then $B=\End_{X_A}(\pi_{X_A*}(P'))\in\Br(A)$ and $u': X\to B$ is represented by $\pi_{X_A*}(P')$, namely $u'=\phi_*(P')\in K_0(X, B)$ where $\phi$ is the projection $X_{A\otimes_k l}\to X$.

The following criterion, which is \cite[Proposition 4.5]{mp}, checks when a toric model is isomorphic to an \'{e}tale algebra in $\C$:
\begin{prop}\label{toric model}
Let $X^*$ be a smooth projective toric model over $k$ that splits over $l$ and $G=\Gal(l/k)$. If $K_0(X^*_l)$ is a permutation $G$-module, then $X^*\cong\Hom_G(P, l)$ in the motivic category $\C$ for any permutation $G$-basis $P$ of $K_0(X^*_l)$.
\end{prop}
\begin{rem}
In particular, this implies that for any split smooth projective toric variety $Y$ over $k$, $Y\cong k^n$ in $\C$ where $n$ equals to the rank of $K_0(Y)$ (also equals to the number of maximal cones of the fan). Note that a smooth projective toric variety $Y$ over $k$ where the fan of $Y_l$ has no symmetry (i.e, $\Aut_{\Sigma}(Y_l)$ is trivial) is automatically split. 
\end{rem}
\begin{lem}\label{lb}
Let $X^*, G$ be the same as before. Then there is an isomorphism $u: X^*\to A$ in $\C$ where $A$ is an \'{e}tale k-algebra and $u$ is represented by an element $Q\in\Pic(X^*_A)$ if and only if $K_0(X^*_l)$ has a permutation $G$-basis of line bundles on $X^*_l$.
\end{lem}
\begin{proof}
$\Rightarrow$: Decompose $A$ as $\prod_{i=1}^t k_i$ where $k_i$ are finite separable field extensions of $k$. We have $X^*_A=\coprod_{i=1}^t X^*_{k_i}$ the disjoint union of $X^*_{k_i}$ and $Q=\coprod_{i=1}^t Q_i$ where $Q_i$ are line bundles on $X^*_{k_i}$. Let $q_i: X^*_{k_i}\to X^*$ be the projections. Then $u=\bigoplus_{i=1}^t q_{i*}Q_i$. Let $p_i: X^*_{k^s}\to X^*_{k_i}$ be the projections and $G_i=\Gal(k_i/k)$. Then 
\[u_{k_s}=\bigoplus_{i=1}^t p_i^*q_i^*q_{i*}(Q_i)=\bigoplus_{i=1}^t\bigoplus_{g\in G_i} p_i^*(gQ_i)\] 
and $A_{k^s}\cong (k^s)^n$ where $n=\sum _{i=1}^t |G_i|$. View $u$ as $u^{\op}: A^{\op}=A\to X^*$. Then the map $u_{k^s}^{\op}$ induces an isomorphism $K_0((k^s)^n)\to K_0(X^*_{k^s})$ where the canonical basis of the former  is sent to $\{p_i^*(gQ_i)\,|\, g\in G_i,1\leqslant i\leqslant t\}$ and this set gives a permutation $\Gamma$-basis of $K_0(X^*_{k^s})$ consisting of line bundles. As $\Gal(k^s/l)$ acts trivially on $K_0(X^*_{k^s})$, this basis descends to $X^*_l$.

$\Leftarrow$: Assume $P$ is a permutation $G$-basis of $K_0(X^*_l)$ consisting of line bundles on $X^*_l$ and $P$ divides into $t$ $G$-orbits. Let $\{S_i\}_{i=1}^t$ be the set of representatives of $G$-orbits, and let $\Gal(l/k_i)$ be the stabilizer of $S_i$. Set $A=\Hom_G(P,l)$. Then $A\cong\prod_{i=1}^t k_i$. Since $X^*$ has a rational point, by \cite[Proposition 5.1]{ctkarm}, we have $S_i \in \Pic(X^*_l)^{\Gal(l/k_i)} \cong \Pic(X^*_{k_i})$, namely $S_i\cong p_i^*(Q_i)$ for some $Q_i\in\Pic(X^*_{k_i})$ where $p_i: X^*_l\to X^*_{k_i}$ are the projections. There is a morphism $u: X^*\to A$ which is represented by $\coprod_{i=1}^t Q_i\in\Pic(X^*_A)$, and by construction, the map $u_l$ induces an isomorphism $K_0(X^*_l)\cong K_0(A_l)$. Using the following lemma, we have $u$ is an isomorphism.
\end{proof}
\begin{lem}
Let $X^*$ be the same as before and $A$ an \'{e}tale $k$-algebra. If $u:X^*\to A$ is a morphism in $\C$ such that $K_0(u_{k^s}): K_0(X^*_{k^s})\to K_0(A_{k^s})$ is an isomorphism, then so is $u$.
\end{lem}
\begin{proof}
There is a commutative diagram:
\begin{center}
\begin{tikzcd}[column sep=large]
K_0(X^*) \arrow{r}{K_0(u)} \arrow{d}
& K_0(A) \arrow{d}\\
K_0(X^*_{k^s})^{\Gamma} \arrow{r}{K_0(u_{k^s})} & K_0(A_{k^s})^{\Gamma}
\end{tikzcd}
\end{center}

The right vertical map is an isomorphism because $A$ is \'{e}tale and so is $K_0(u_{k^s})$ by assumption. The left vertical map is an isomorphism by \cite[Corollary 5.8]{mp}. Thus, $K_0(u)$ is also an isomorphism.

Write $w=u^{\op}: A\to X^*$. Then by the splitting principle (their Proposition $6.1$) and the proof, $K_0^{X^*}(w): K_0(X^*,A)\to K_0(X^*\times X^*)$ is surjective. Thus, there exists $v\in K_0(X^*,A): X^*\to A$ such that $w\circ v=K_0^{X^*}(w)(v)=1_{X^*}$, and thus $K_0(w\circ v)=K_0(w)K_0(v)=1_{K_0(X^*)}$. Since $K_0(w)=\phi$ is an isomorphism, we have $K_0(v)=\phi^{-1}$ and $K_0(v\circ w)=K_0(v)K_0(w)=1_{K_0(A)}$. This implies $v\circ w=1_A$ and thus $v$ is a two sided inverse of $w$ in $\C$.
\end{proof}
The proof of $(3)\Leftrightarrow(4)$ in their Proposition 7.9 shows that the $T_l$-isomorphism $f: X_l\to X^*_l$ induces a $G=\Gal(l/k)$-module isomorphism $f^*: K_0(X^*_l)\to K_0(X_l)$. Thus, $K_0(X^*_l)$ has a permutation $G$-basis of line bundles on $X^*_l$ if and only if $K_0(X_l)$ has such a basis. Note that the proof $(1)\Rightarrow(2)$ (an isomorphism $u: X^*\to A$ gives an isomorphism $u': X\to B$), which uses the construction (\ref{endo}) recalled at the beginning of the section, works only when $u$ is represented by an element $Q\in\Pic(X^*_A)$. Thus, we have the following instead:
\begin{thm}\label{basis}
Let $X$ be a smooth projective toric $T$-variety over $k$ that splits over $l$ and $G=\Gal(l/k)$. Assume $K_0(X_l)$ has a permutation $G$-basis $P$ of line bundles on $X_l$. Let $\{P_i\}_{i=1}^t$ be $G$-orbits of $P$, and let $\pi: X_l\to X$ be the projection. For any $S_i\in P_i$, set $B_i=\End_{\O_X}(\pi_*(S_i))$ and $B=\prod_{i=1}^t B_i$. Then the map $u=\bigoplus_{i=1}^t \pi_*(S_i): X\to B$ gives an isomorphism in the motivic category $\C$.
\end{thm}
\begin{proof}
By Lemma \ref{lb}, we have an isomorphism $u: X^*\to A$ represented by $Q\in\Pic(X^*_A)$. Here $A\cong\prod_{i=1}^t k_i$ where $\Gal(l/k_i)$ are the stabilizers of $S_i$ under the $G$-action. Then $Q$ is the disjoint union $\coprod_{i=1}^t Q_i$ where the $Q_i\in\Pic(X^*_{k_i})$ descend from $(f^*)^{-1}(S_i)\in \Pic(X^*_l)^{\Gal(l/k_i)}$. Now we run the construction (\ref{endo}) for $Q_i$:
\begin{center}
\begin{tikzcd}
X_{k_i\otimes_k l} \arrow{r}{f_i} \arrow{d}{\pi_X} & X^*_{k_i\otimes_k l} \arrow{d}{\pi_{X^*}}\\
X_{k_i}  & X^*_{k_i}
\end{tikzcd}
\end{center}

Let $p: X_{l}\to X_{k_i}$ and $q: X_{k_i}\to X$ be the projections. Then $\pi_{X*}f_i^*\pi_{X^*}^*(Q_i)\cong p_*(S_i)\otimes_k k_i$ where its $\O_{X_{k_i}}$-module structure comes from the one on $p_*(S_i)$. Thus, $\End_{\O_{X_{k_i}}}(\pi_{X*}f_i^*\pi_{X^*}^*(Q_i))\cong\End_{\O_{X_{k_i}}}(p_*(S_i))\otimes_k \End_k(k_i)$ is Brauer equivalent to $B_i'=\End_{\O_{X_{k_i}}}(p_*S_i)$. It remains to prove that $B_i\cong B_i'$. There is a $G$-isomorphism:
$B_i\otimes_k l \cong\End_{\O_{X_l}}(\pi^*\pi_*(S_i))\cong\End_{\O_{X_l}}(p^*q^*q_*p_*(S_i))\cong \End_{\O_{X_l}}(p^*p_*(S_i)\otimes_k k_i)\cong\End_{\O_{X_l}}(p^*p_*(S_i))\otimes_k k_i\cong (B_i'\otimes_{k_i}l)\otimes_k k_i\cong B_i'\otimes_k l$.
The fourth isomorphism follows from Lemma \ref{vanish}. Take $G$-invariants on both sides, we have $B_i\cong B_i'$.
\end{proof}
\begin{lem}\label{vanish}
Let $X$ be a proper variety over $k$ and assume that there is a finite group $G$ acting on Cartier divisors $\CDiv(X)$. Let $D\in \CDiv(X)$ and $g\in G$ such that $D$ and $gD$ are not linearly equivalent. Then $\Hom_{\O_X}(\O_X(D),\O_X(gD)) = 0$. 
\end{lem}
\begin{proof}
Assume that $\Hom_{\O_X}(\O_X(D),\O_X(gD))\neq 0$, which is equivalent to $\O_X(gD-D)$ having a nonzero global section $s$. Since $G$ is a finite group, $g^n=1$ for some $n$. Therefore, the invertible sheaf $\O_X(D-gD)=(g^{n-1}\otimes\cdots\otimes g\otimes 1)\O_X(gD-D)$ has a nonzero global section $t=g^{n-1}s\otimes\cdots\otimes s$. View $s,t$ as maps $s:\O_X(D)\to \O_X(gD)$ and $t:\O_X(gD)\to \O_X(D)$. Since $st,ts\in\Gamma(X,O_X)=k$ are nonzero, we have $\O(gD-D)\cong\O_X$, a contradiction.
\end{proof}
\begin{rem}
There is a more ``economical" description of the algebra isomorphic to $X$ in $\C$:

Write $S_i=\O(-D_i)$ where the $D_i$ are torus invariant. Let $\Gal(l/l_i)$ be the stabilizers of $D_i$ under the $G$-action and let $\pi_i:X_{l_i}\to X$ be the projections. Then divisors $D_i$ and thus invertible sheaves $S_i$ descend to $X_{l_i}$, and we use the same notation. Then $X\cong\prod_{i=1}^t \End_{\O_X}(\pi_{i*}(S_i))$. In effect, it replaces all $\M_n(k)$ in $B$ constructed in the theorem by $k$ which is an isomorphism in $\C$. 
\end{rem}
\begin{rem}
A question remains: If $K_0(X_l)$ is a permutation $G$-module, can we always find a permutation $G$-basis of line bundles?
\end{rem}
Recall that for $n\geqslant 0$, $K_n$ defines a functor $K_n:\C\to\textit{Ab}$. Hence we have
\begin{cor}
$K_n(X)\cong \prod_{i=1}^t K_n(B_i)$.
\end{cor}
\section{Separable Algebras for Toric Surfaces}\label{sa-surface}
\subsection{Separable algebras for minimal toric surfaces}
Recall the families of minimal toric surfaces described in Theorem \ref{K0-surface}: Let $X$ be a minimal smooth projective toric $T$-surface over $k$ that splits over $l$, and let $X^*$ be its associated toric model. Let $\pi: X_l\to X$ be the projection. All isomorphisms below are taken in the motivic category $\C$.
\begin{description}
\item[\hyperlink{i'}{(i)}\hypertarget{i''}{}] If $X_l\cong F_a, a\geqslant 2$, then $X^*\cong k^4$ and $X\cong k\times Q\times k\times Q$, where $Q\cong\End_{\O_X}(\pi_*J_1)$ is a quaternion $k$-algebra.
\item[\hyperlink{ii'}{(ii)}\hypertarget{ii''}{}] More generally, let $X=\text{SB}(A)$ be a Severi-Brauer variety of dimension $n$ and $J=\O_{X_l}(-1)$. Then $X^*\cong k^{n+1}$ and $X\cong k\times\prod_{i=1}^n A^{\otimes i}$ where $A^{\otimes i}\cong\End_{\O_X}(\pi_*J^i)$, see Example \ref{sb}.
\item[\hyperlink{iii'}{(iii)}\hypertarget{iii''}{}] If $X_l\cong\P^1\times\P^1$, then $X^*\cong k\times K\times k$ where $K$ is a quadratic \'{e}tale algebra and the discriminant extension of $X$, and $X\cong k\times B\times A$, where $B\cong\End_{\O_X}(\pi_*J_1)$ is an Azumaya $K$-algebra of rank $4$ and $A\cong\End_{\O_X}(\pi_*(J_1J_2))$ is a central simple $k$-algebra of degree $4$, see Example \ref{inv}.
\item[\hyperlink{iv'}{(iv)}\hypertarget{iv''}{}] See Example \ref{dp} where $X^*\cong k\times K\times L$ and $P\cong\End_{\O_X}(\pi_*R_1)$ and $Q\cong\End_{\O_X}(\pi_*Q_1)$.
\end{description}

Now let $X$ be a smooth projective toric $T$-variety over $k$ that splits over $l$ and $G=\Gal(l/k)$. Recall that $X$ is uniquely determined by the associated toric model $X^*$, which corresponds to $\rho: \Gamma\to\GL(n,\Z)$, the fan $\Sigma$ such that $\rho(\Gamma)\subseteq\Aut_{\Sigma}$, and a principal homogeneous space $U\in H^1(k,T)$. Every variety within a family above has the same fan. Let $\rho': G\hookrightarrow\Aut_{\Sigma}(X_l)$ be the inclusion induced by $\rho$. We want to see how the separable algebras described above relate to $\rho'$ and $U$.

Let $\dim X=n$ and let $N$ be the number of rays in the fan $\Sigma$. Then the Picard rank of $X_l$ is $m=N-n$. Write $M$ for the group of characters of $T_l$ and $\CDiv_{T_l}$ for $T_l$-invariant Cartier divisors. There is a natural action of $\Aut_{\Sigma}(X_l)$ on $M$ and $\CDiv_{T_l}(X_l)$, and an induced action on $\Pic(X_l)$ via the canonical morphism $\CDiv_{T_l}(X_l)\to\Pic(X_l)$, $D\mapsto \O_{X_l}(D)$.

We have a short exact sequence of $\Aut_{\Sigma}(X_l)$-modules and therefore of $G$-modules via $\rho'$:
\begin{equation}\label{toricex}
0\to M\to \CDiv_{T_l}(X_l)\to \Pic(X_l)\to 0,
\end{equation}
or simply $0\to\Z^n\to\Z^N\to\Z^m\to 0.$
It corresponds to the short exact sequence of tori over $l$:
\begin{equation*}
1\to \G_{m,l}^m\to \G_{m,l}^N\to \G_{m,l}^n\to 1
\end{equation*}
and the sequence descends to
\begin{equation}\label{tori}
1\to S\to V\to T\to 1.
\end{equation}

Let $i: \Aut_{\Sigma}(X_l)\hookrightarrow S_N$ where $S_N$ is the group of permutations of the canonical $\Z$-basis of the lattice $\Z^N$ and it induces $i_*: H^1(G,\Aut_{\Sigma})\to H^1(G, S_N)$. Let $[\alpha]=i_*[\rho']$ and let $E$ be the corresponding \'{e}tale $k$-algebra of degree $N$. Then $V=\R_{E/k}(\G_{m,E})$. Let $j: \Aut_{\Sigma}(X_l)\to\GL(m,\Z)$ be the map induced by the action of $\Aut_{\Sigma}(X_l)$ on $\Pic(X_l)$ which induces $j_*: H^1(G, \Aut_{\Sigma})\to H^1(G, \GL(m,\Z))$. Let $[\beta]=j_*[\rho']$. Then $S$ is the torus corresponding to $[\beta]$.

The short exact sequence of tori over $k$ gives
\begin{equation*}
0\to H^1(G,T)\overset{\delta}{\to} H^2(G,S)\to\Br(E).
\end{equation*}
Here $H^1(G,V)=H^1(G,\R_{E/k}(\G_{m,E})(l))=\prod H^1(\Gal(E_t/k),E_t^{\times})=0$ by Hilbert $90$ Theorem where $E=\prod E_t$ and the $E_t$ are finite separable field extensions of $k$.

Let $S^*=\Hom(S_l, G_{m,l})$ be the group of characters over $l$. Then sequence (\ref{toricex}) can be rewritten as
\begin{equation*}
0\to T^*\to V^*\to S^*\to 0
\end{equation*}
which induces $H^0(G, S^*)\overset{\partial}{\to} H^1(G, T^*)$. Geometrically, $\partial$ is the map $\Pic(X^*)\to\Pic(T)$ which sends $Q\in\Pic(X^*)$ to its restriction $Q|_T$ on $T$. 

There is a $G$-equivariant bilinear map $S(l)\otimes S^*\to l^{\times}$ which sends $x\otimes \chi$ to $\chi(x)$, and it induces a pairing of Galois cohomology groups $\cup: H^2(G,S)\otimes H^0(G, S^*)\to \Br(k)$. Similarly, we have $\cup: H^1(G,T)\otimes H^1(G, T^*)\to \Br(k)$.
\begin{lem}\label{comm}
The following diagram is commutative:
\begin{center}
\begin{tikzcd}
H^1(G, T)\otimes H^0(G, S^*) \arrow{r}{1\otimes\partial} \arrow{d}{\delta\otimes 1}
& H^1(G,T)\otimes H^1(G, T^*) \arrow{d}{\cup}\\
H^2(G,S)\otimes H^0(G, S^*) \arrow{r}{\cup} & \Br(k)
\end{tikzcd}
\end{center}
\end{lem}
\begin{proof}
Let $a\in H^1(G,T), \varphi\in H^0(G,S^*)$. For each $a_g\in T(l), g\in G$, pick $b_g\in V(l)$ that maps to $a_g$. Then $(\delta a)_{g,h}=b_{gh}^{-1}b_g{}^g b_h, g,h\in G$. Pick $\phi\in V^*$ that maps to $\varphi$. Then $(\partial \varphi)_g={\phi}^{-1}{}^g \phi$. Let $\alpha=a\cup(\partial\varphi)$ and $\beta=(\delta a)\cup \varphi$. Then $\alpha_{g,h}={}^g(\partial\varphi)_h(a_g)={}^g(\phi^{-1}{}^h\phi)(b_g)=({}^g\phi^{-1})(b_g)\cdot({}^{gh}\phi)(b_g)$ and $\beta_{g,h}=({}^{gh}\varphi)((\delta a)_{g,h})=({}^{gh}\phi)(b_{gh}^{-1})\cdot({}^{gh}\phi)(b_g)\cdot({}^{gh}\phi)({}^g b_h)$. Set $\theta_g=({}^g \phi)(b_g)$. Then $\beta_{g,h}=\theta_{gh}^{-1}\theta_g{}^g\theta_h \alpha_{g,h}$. Thus, $\alpha$ and $\beta$ give the same cycle class in $\Br(k)$.
\end{proof}
Let $P\in\Pic(X_l)$ be a line bundle on $X_l$ with stabilizer group $\Gal(l/\kappa)$ under the $G$-action. Since $P \in \Pic(X_l)^{\Gal(l/\kappa)}\cong (S^*)^{\Gal(l/\kappa)}$, the line bundle $P$ corresponds to a character $\chi: S_{\kappa} \to \G_{m,\kappa}$ over $\kappa$, or equivalently $\chi': S\to \R_{\kappa/k}(\G_{m,\kappa})$. Let $\pi: X_l\to X$ be the projection.
\begin{prop}\label{geom}
Let $\delta_P: H^1(G,T)\overset{\delta}{\to} H^2(G,S)\overset{\chi'}{\to} \Br(\kappa)$ be the composition map. Then $\delta_P[U]=[\End_{\O_X}(\pi_*P)]\in\Br(\kappa)$.
\end{prop}
\begin{proof}
First we prove the case when $\kappa=k$. In this case, the line bundle $P\in\Pic(X_l)^G\cong\Pic(X^*)$. Thus, there is $Q\in\Pic(X^*)$ such that $P\cong f^*\pi_{X^*}^*Q$ where $\pi_{X^*}: X^*_l\to X^*$ is the projection and $f: X_l\to X^*_l$ is the $T_l$-isomorphism. \cite[Lemma 7.3]{mp} shows that $[U]\cup[Q|_T]=[\End_{\O_X}(\pi_*P)]\in\Br(k)$. On the other hand, $\delta_P([U])=\delta[U]\cup[\chi']=\delta[U]\cup[Q]$. By Lemma \ref{comm}, $\delta_P([U])=[U]\cup[\partial Q]=[U]\cup[Q|_T]$. 

In general, let $H=\Gal(l/\kappa)$ and consider the restriction map $\Res: H^1(G, T)\to H^1(H, T_{\kappa})$ which sends $[U]$ to $[U_{\kappa}]$. There is a commutative diagram:
\begin{center}
\begin{tikzcd}
H^1(G,T) \arrow{r}{\delta} \arrow{d}{\Res} & H^2(G, S) \arrow{r}{\chi'}\arrow{d}{\Res} & \Br(\kappa)\arrow[equal]{d}\\
H^1(H,T_{\kappa}) \arrow{r}{\delta} & H^2(H,S_{\kappa}) \arrow{r}{\chi}& \Br(\kappa)
\end{tikzcd}
\end{center}
Thus, $\delta_P[U]=[\End_{\O_{X_{\kappa}}}(\pi_{\kappa*}P)]$ where $\pi_{\kappa}: X_l\to X_{\kappa}$ is the projection. By the proof of Lemma \ref{lb}, $\End_{\O_{X_{\kappa}}}(\pi_{\kappa*}P)\cong \End_{\O_X}(\pi_*P)$.
\end{proof}
\begin{cor}\label{pic}
Let $X$ be a smooth projective toric variety over $k$ that splits over $l$ and $G=\Gal(l/k)$. Assume $\Pic(X_l)$ is a permutation $G$-module, i.e, the torus $S$ is quasi-trivial and thus has the form $\prod_{i=1}^t \R_{k_i/k}\G_{m,k_i}$ where $k_i$ are finite separable field extensions of $k$. Then the principal homogeneous space $U$ is uniquely determined by $(B_i\in\Br(k_i))_{1\leqslant i\leqslant t}$ where $B_i$ split over $E$. Let $\{S_i\}_{i=1}^t$ be the set of representatives for $G$-orbits of $\Pic(X_l)$. Then $B_i$ comes from $\End_{\O_X}(\pi_*S_i)$.
\end{cor}
\begin{proof}
The result follows from Proposition \ref{geom} and the exact sequence 
\[0\to H^1(k,T)\to\prod_{i=1}^t \Br(k_i)\to\Br(E).\]
\end{proof}
\begin{rem}
Families \hyperlink{i'}{(i)}\hyperlink{ii'}{(ii)}\hyperlink{iii'}{(iii)} and their blow-ups have permutation Picard groups.
\end{rem}
$\bf{\hyperlink{ii''}{(ii)}}$: $X=\text{SB}(A)$ is a Severi-Brauer variety of dimension $n$, $\Aut_{\Sigma}(X_l)=S_{n+1}$. We have
\begin{equation*}
1\to \G_{m,k}\to \R_{E/k}(\G_{m,E})\to T\to 1
\end{equation*}
which induces
\begin{equation*}
0\to H^1(G,T)\overset{\delta}{\to} \Br(k)\to \Br(E).
\end{equation*}
Then $\delta(U)=[A]$ and $A$ splits over $E$, see \cite[Example 8.5]{mp}.

$\bf{\hyperlink{i''}{(i)}}$: $X_l=F_a, a\geqslant 2, \Aut_{\Sigma}=S_2$ and $E$ factors as $k\times F\times k$ where $F$ is the quadratic \'{e}tale $k$-algebra corresponding to $[\rho']\in H^1(G,S_2)$. We have
\begin{equation*}
1\to \G_{m,k}\to \G_{m,k}\times \R_{F/k}(\G_{m,F})\to T\to 1
\end{equation*}
where $\G_{m,k}\to\G_{m,k}$ is the $a$-th power homomorphism. It induces
\begin{equation*}
0\to H^1(G,T)\overset{\delta}{\to} \Br(k)\to \Br(k)\times \Br(F)
\end{equation*}
where $[U]\mapsto [Q]\mapsto ([Q^{\otimes a}],[Q_F])$. By Lemma \ref{cb}, the toric surface $X$ is a $\P^1$-bundle over some conic curve $C$. We have the torus of $C$ is $T'=\R_{F/k}(\G_{m,F})/\G_{m,k}$. There is a commutative diagram with exact rows:
\begin{center}
\begin{tikzcd}
1 \arrow{r} & \G_{m,k} \arrow{r}\arrow[equal]{d} & \G_{m,k}\times \R_{F/k}(\G_{m,F}) \arrow{r}\arrow{d} & T \arrow{r}\arrow{d}{h} & 1\\
1 \arrow{r} & \G_{m,k}\arrow{r} & \R_{F/k}(\G_{m,F})\arrow{r} & T' \arrow{r} & 1\\
\end{tikzcd}
\end{center}
Hence, the image of $[U]$ under $\delta\circ h_*: H^1(G,T)\to H^1(G, T')\to \Br(k)$ is $[Q]$, and thus $C=\text{SB}(Q)$. Since a quaternion algebra has a period at most $2$ in the Brauer group, if $a$ is odd, then $[Q^{\otimes a}]\in\Br(k)$ being trivial implies that $Q=\M_2(k)$. Thus we have
\begin{prop}
Let $X$ be a toric surface that is a form of $F_{2a+1}$. Then $X\cong F_{2a+1}$.
\end{prop}
\begin{rem}
Iskovskih showed that any form of $F_{2a+1}$ is trivial \cite[Theorem 3(2)]{isksurf}. The above proposition reproves this result in the case of toric surfaces.
\end{rem}
$\bf{\hyperlink{iii''}{(iii)}}$: Let $X_l=\P^1\times\P^1, \Aut_{\Sigma}=D_8$. In this case, the map $\beta: G\to\GL(2,\Z)$ factors through $\gamma:G\to S_2$ where $S_2$ permutes $\O(1,0)$ and $\O(0,1)$. Then the quadratic \'{e}tale algebra $K$ corresponds to $\gamma$. We have
\begin{equation*}
1\to \R_{K/k}(\G_{m,K})\to \R_{E/k}(\G_{m,E})\to T\to 1
\end{equation*}
which induces
\begin{equation*}
0\to H^1(G,T)\overset{\delta}{\to} \Br(K)\to \Br(E).
\end{equation*}
Then $\delta(U)=[B]$ and $B$ splits over $E$. Let $\text{N}_{K/k}:\R_{K/k}(\G_{m,K})\to \G_{m,k}$ be the norm map which induces $\text{cor}_{K/k}: \Br(K)\to\Br(k)$. Then $[A]=\text{cor}_{K/k}[B]$.
\subsection{Separable algebras for toric surfaces} 
Let $X$ be a smooth projective toric $T$-surface over $k$ that splits over $l$ and $G=\Gal(l/k)$. Recall that we have a finite chain of blow-ups of toric $T$-surfaces
\begin{equation*}
X=X_0\to X_1\to\cdots\to X_n=X'
\end{equation*}
where $X'$ is minimal. For $1\leqslant i\leqslant n$, let $f_i: (X_{i-1})_l\to (X_i)_l$ which are the blow-ups of $G$-sets of disjoint $T_l$-invariant points. Let $E_i$ be the $G$-sets of the exceptional divisors of $f_i$ and $X'\cong B$ in $\C$.
\begin{prop}
$X\cong B\times \prod_{i=1}^n \Hom_G(E_i, l)$ in $\C$.
\end{prop}
\begin{proof}
We only need to consider the simple case: Let $f: Y\to Z$ be a blow-up of toric $T$-surfaces  and let $E=\{P_j\}$ be the $G$-set of line bundles associated to the exceptional divisors of $g=f_l$. We assume further that the $G$-action on $E$ is transitive. 

Let $p:Y_l\to Y$ and $q: Z_l\to Z$ be the projections. Then we have a commutative diagram:
\begin{center}
\begin{tikzcd}
Y_l \arrow{r}{g} \arrow{d}{p}
& Z_l \arrow{d}{q}\\
Y \arrow{r}{f} & Z
\end{tikzcd}
\end{center}

Recall that if $K_0(Z_l)$ has a $G$-basis $\gamma$, then $g^*(\gamma)\cup E$ is a $G$-basis of $K_0(Y_l)$. Since $Z$ is a toric surface, we can assume $\gamma$ consists of line bundles over $Z_l$. Let $P \in\gamma$. Then
\begin{equation*}
\End_{\O_Y}(p_*g^*P)\cong \End_{\O_Y}(f^*q_*P)\cong \Hom_{\O_Z}(q_*P,f_*f^*(q_*P))\cong \End_{\O_Z}(q_*P)
\end{equation*}
where $f_*f^*$ is identity because $f$ is flat proper and $f_*\O_Y=\O_Z$.

As for the $G$-orbit $E$, we have $\bigoplus_j P_j=p^*Q$ for some locally free sheaf $Q$ on $Y$. By Lemma \ref{vanish} and the assumption that $G$ acts transitively on $E$, we have $\End_{\O_Y}(Q)\cong \Hom_G(E,l)$. It is Brauer equivalent to $\End_{\O_Y}(p_*P_j)$ for any $P_j\in E$. Thus the result follows from Theorem \ref{basis}.
\end{proof}
\section{Derived categories of toric surfaces}\label{dcat}
Let $X$ be a smooth projective variety over $k$ and let $D^b(X)$ be the bounded derived category of coherent sheaves on $X$. We will define exceptional objects and collections in a generalized way.
\begin{defn}
Let $A$ be a finite simple $k$-algebra. An object $V$ in $D=D^b(X)$ is called \textit{$A$-exceptional} if $\Hom_D (V,V)=A$ and $\Ext^i_D(V,V)=0$ for $i\neq 0$.
\end{defn}
\begin{defn}
A set of objects $\{V_1,\dots, V_n\}$ in $D=D^b(X)$ is called an \textit{exceptional collection} if for each $1\leqslant i\leqslant n$, the object $V_i$ is $A_i$-exceptional for some finite simple $k$-algebra $A_i$, and $\Ext^r_D(V_i,V_j)=0$ for any integer $r$ and $i>j$. The collection is \textit{full} if the thick triangulated subcategory $\<V_1,\dots, V_n\>$ generated by the $V_i$ is equivalent to $D^b(X)$.
\end{defn}
\begin{defn}
A set of objects $\{V_1,\dots, V_n\}$ in $D\in D^b(X)$ is called an \textit{exceptional block} if it is an exceptional collection and $\Ext^r_D(V_i,V_j)=0$ for any integer $r$ and $i\neq j$. Note that the ordering of the $V_i$ in this case does not matter.
\end{defn}
Assume $\{V_1,\dots, V_n\}$ is a full exceptional collection as above. Since $\<V_i\>$ is equivalent to $D^b(A_i)$, the bounded derived category of right $A_i$-modules, we have semiorthogonal decompositions $D^b(X)=\<V_1,\dots,V_n\>=\<D^b(A_1),\dots, D^b(A_n)\>$.

The semiorthogonal decomposition of $D^b(X)$ can be lifted to the world of dg categories. For details about dg categories, see \cite{dgcat}. There is a dg enhancement of $D^b(X)$, denoted as $D^b_{dg}(X)$ where $D^b_{dg}(X)$ is the dg category with same objects as $D^b(X)$ and whose morphisms have a dg $k$-module structure such that $H^0(\Hom_{D^b_{dg}(X)}(x,y))=\Hom_{D^b(X)}(x,y)$. Let $\perf_{dg}(X)$ be the dg subcategory of perfect complexes. Since $X$ is smooth projective, $\perf_{dg}(X)$ is quasi-equivalent to $D^b_{dg}(X)$. For an $A$-exceptional object $V$, the pretriangulated dg subcategory $\<V\>_{dg}$ generated by $V$ is quasi-equivalent to $D^b_{dg}(A)$. Therefore, there is a dg enhancement of the semiorthogonal decomposition $D^b_{dg}(X)=\<V_1,\dots, V_n\>_{dg}$, which is quasi-equivalent to $\<D^b_{dg}(A_1),\dots,D^b_{dg}(A_n)\>_{dg}$. 

Let $dgcat$ be the category of all small dg categories. There is a universal additive functor $U: dgcat\to Hmo_0$ where $Hmo_0$ is the category of noncommutative motives, see \cite[\S2.1-2.4]{noncomm}. We have $U(\perf_{dg}(X))\simeq\bigoplus_{i=1}^n U(D^b_{dg}(A_i))\simeq\bigoplus_{i=1}^n U(A_i)$. On the other hand, the motivic category $\C$ is a full subcategory of $Hmo_0$ by sending a pair $(X,A)$ to $\perf_{dg}(X,A)$, the dg category of complexes of right $\O_X\otimes_k A$-modules which are also perfect complexes of $\O_X$-modules \cite[Theorem 6.10]{tamp} or \cite[Theorem 4.17]{noncomm}. The above discussion gives the following well-known fact:
\begin{thm}\label{derivedtomotivic}
Let $X$ be a smooth projective variety over $k$. If $D^b(X)$ has a full exceptional collection of objects $\{V_1,\dots,V_n\}$ where each $V_i$ is $A_i$-exceptional, then $X\cong\prod_{i=1}^n A_i$ in the motivic category $\C$.
\end{thm}
We know for toric varieties satisfying the conditions of Theorem \ref{basis}, they have a complete motivic decomposition into central simple algebras. The following lemma gives a criterion when the motivic decomposition can be lifted to the decomposition of the derived category (i.e, the reverse of Theorem \ref{derivedtomotivic}):
\begin{lem}\label{ordering}
Let $X$ be a smooth projective toric variety over $k$ that splits over $l$ and $G=\Gal(l/k)$. Assume $K_0(X_l)$ has a permutation $G$-basis $P$ of line bundles over $X_l$. Let $\{P_i\}_{i=1}^t$ be $G$-orbits of $P$ and let $\pi: X_l\to X$ be the projection. Assume each $G$-orbit $P_i$ is an exceptional block. If there is an ordering for $G$-orbits $\{P_i\}_{i=1}^t$ such that $\{P_1,\dots, P_t\}$ gives a full exceptional collection of $D^b(X_l)$, then for any $S_i\in P_i$, the set $\{\pi_*S_1, \dots, \pi_*S_t\}$ is a full exceptional collection of $D^b(X)$.
\end{lem}
\begin{proof}
First we show that $\{\pi_*S_1,\dots, \pi_*S_t\}$ is an exceptional collection. Since $\pi$ is flat and finite, both $\pi^*: D^b(X)\to D^b(X_l)$ and $\pi_*: D^b(X_l)\to D^b(X)$ are exact functors. The result follows from $\Ext^r_{D^b(X)}(\pi_*S_i, \pi_* S_j)\otimes_k l\cong \Ext^r_{D^b(X_l)}(\pi^*\pi_*S_i, \pi^*\pi_*S_j)\cong \bigoplus_{g, g'\in G}\Ext^r_{D^b(X_l)}(gS_i, g'S_j)$. In particular, $\pi_*S_i$ is an exceptional object and thus $\<\pi_*S_i\>$ is an admissible subcategory of $D^b(X)$. Since $\<\pi_*S_i\otimes_k l\>=\<P_i\>$ and $D^b(X_l)=\<P_1,\dots, P_t\>$, by \cite[Lemma 2.3]{aberdel}, we have $D^b(X) = \< \pi_*S_1, \dots, \pi_*S_t\>$.
\end{proof}
Using the classification of toric surfaces, we can confirm the lifting for toric surfaces:
\begin{thm}\label{dersurf}
Let $X$ be a smooth projective toric surface over $k$ that splits over $l$ and $G=\Gal(l/k)$. Then $K_0(X_l)$ has a permutation $G$-basis $P$ of line bundles over $X_l$ such that each $G$-orbit is an exceptional block. Furthermore, there exists an ordering of the $G$-orbits $\{P_i\}_{i=1}^t$ of $P$ such that $\{P_1,\dots, P_t\}$ gives a full exceptional collection of $D^b(X_l)$. Therefore, for any $S_i\in P_i$, $\{\pi_*S_1$, $\dots, \pi_*S_t\}$ is a full exceptional collection of $D^b(X)$ where $\pi: X_l\to X$ is the projection.
\end{thm} 
\begin{proof}
First assume that $X$ is minimal. By the classification of minimal toric surfaces (Theorem \ref{ms}),  we have $X_l$ is (i) $F_a, a\geqslant 2$; (ii) $\P^2$; (iii) $\P^1\times \P^1$; (iv) del Pezzo surface of degree $6$. Using the notation introduced in Theorem \ref{K0-surface}, the derived category $D^b(X_l)$ has the following full exceptional collections of line bundles:

\hyperlink{i'}{(i)} $\{\O, \O(D_1), \O(D_2), \O(D_1+D_2)\}$;

\hyperlink{ii'}{(ii)} $\{\O, \O(D_1), \O(2D_1)\}=\{\O, \O(1), \O(2)\}$;

\hyperlink{ii'i}{(iii)} $\{\O, \O(D_1), \O(D_2), \O(D_1+D_2)\}=\{\O, \O(1,0), \O(0,1), \O(1,1)\}$;

\hyperlink{iv'}{(iv)} $\{\O, R_1^{\vee}, R_2^{\vee}, R_3^{\vee}, Q_1^{\vee}, Q_2^{\vee}\}$ where $(-)^{\vee}$ is the dual of the invertible sheaf.

(i)-(iii) follow from the projective bundle theorem \cite[Theorem 2.6]{orlproj} and (iv) follows from \cite[Proposition 9.1]{aberdel} or \cite{bssdel}. Moreover, the collections $\{\O(1,0),\O(0,1)\}$, $\{R_i^{\vee}\}_{i=1}^3$ and $\{Q_j^{\vee}\}_{j=1}^2$ are exceptional blocks. These sets are the only $G$-orbits with more than one object. Therefore, each $G$-orbit is an exceptional block.

Now it suffices to consider the case that $f: X\to X'$ is a simple blow-up of a minimal toric surface $X'$, that is, the map $f_l: X_l\to X'_l$ is the blow-up of a $G$-set of disjoint torus invariant points of $X'_l$ where $G$ acts on the set transitively. Let $E_i$ be the exceptional divisors of $f_l$. Let $E$ be the set $\{\O_{E_i}(-1)\}$. By \cite[Theorem 4.3]{orlproj}, the derived category $D^b(X)$ has a full exceptional collection $\{E, L^\bullet f^*D^b(X')\}$. Note that the full exceptional collections of minimal toric surfaces provided above all have the structure sheaf $\O$ as the first object. The right mutation of the pair $(\O_{E_i}(-1), \O)$ is $(\O, \O(E_i))$ (the extension case in \cite[Proposition 2.3]{karnog}). Therefore, the right mutation of $\{E,\O\}$ is $\{\O, E'\}$ where $E'=\{\O(E_i)\}$. The $G$-orbit $E'$ is an exceptional block because the order in the set is exchangeable. Hence, $D^b(X_l)$ has a full exceptional collection $\{\O, E', \text{the rest of the line bundles provided above}\}$ (they form a basis of $K_0(X_l)$) and each $G$-orbit is an exceptional block.
\end{proof}

\bibliography{tsBib.bib}
\bibliographystyle{alpha}
\end{document}